\documentclass[noindent,12pt]{article}
\usepackage{latexsym,amsmath,amsfonts}
\newtheorem{theorem}{Theorem}
\newtheorem{proposition}{Proposition}

\newtheorem{lemma}{Lemma}

\newcommand{\thmref}[1]{Theorem~\ref{thm:#1}} % Theorem tag equals ``thm''
\newcommand{\lemref}[1]{Lemma~\ref{lem:#1}} % Lemma tag equals ``lem''
\newcommand{\propref}[1]{Proposition~\ref{prop:#1}} % Proposition tag equals ``prop''
\newcommand{\secref}[1]{Section~\ref{sec:#1}} % Section = ``sec''
\newcommand{\eqnref}[1]{(\ref{eq:#1})} % Equation = ``eq''
\def\be{\begin{equation} }
\def\ee{ \end{equation}}

\def\ben{\begin{equation*}}
\def\een{\end{equation*}}
\def\bea{\begin{eqnarray}}
\def\eea{\end{eqnarray}}
\def\ee{\end{eqnarray}}
\def\bean{\begin{eqnarray*}}
\def\eean{\end{eqnarray*}}

\newcommand\ignore[1]{}

\def\R{\mathbb{R}} % real numbers
\def\C{\mathbb{C}} % complex numbers
\def\N{\mathbb{N}} % naturals
\newcommand{\Ex}[1]{\mathbb{E}\left[#1\right]} % \Ex{abc} prints E[abc] with appropriately sized bracketss
\newcommand{\Prwo}{\mathbb{P}} % \Prp{a}{b} yields P_a(b); a is a subscript, e.g. the initial state in a Markov Process

\renewcommand{\Pr}[1]{\mathbb{P}\left(#1\right)} % \Pr{abc} prints P(abc) with appropriately sized parentheses
\newcommand{\theita}[1]{\Theta\left(#1\right)}

\def\sF{\mathcal{F}}

\def\sS{\mathcal{S}}

\newcommand\QED{\ifhmode\allowbreak\else\nobreak\fi
\quad\nobreak$\Box$\medbreak}
\newcommand{\proofstart}{\par\noindent\sl Proof:\rm\enspace}
\newcommand{\proofend}{\QED\par}
\newenvironment{proof}{\proofstart}{\proofend}

\def\eps{\epsilon}
\def\tr{{\rm Tr}}

\begin{document}

\title{Sums of random Hermitian matrices and an inequality by
Rudelson}
\author{Roberto Imbuzeiro Oliveira\thanks{IMPA, Rio de Janeiro, RJ, Brazil, 22430-040. \texttt{rimfo@impa.br}}} \maketitle
\begin{abstract}We give a new, elementary proof of a key inequality used by Rudelson in the derivation of his well-known bound for random sums of rank-one operators. Our approach is based on Ahlswede and Winter's technique for proving operator Chernoff bounds. We also prove a concentration inequality for sums of random matrices of rank one with explicit constants.\end{abstract}

\section{Introduction}

This note mainly deals with estimates for the operator norm $\|Z_n\|$ of random sums
\begin{equation}\label{eq:defZ_n}Z_n\equiv \sum_{i=1}^n\eps_iA_i\end{equation}
of deterministic Hermitian matrices $A_1,\dots,A_n$ multiplied by random coefficients. Recall that a {\em Rademacher sequence} is a sequence $\{\eps_i\}_{i=1}^n$ of i.i.d. random variables with $\eps_1$ uniform over $\{-1,+1\}$. A
{\em standard Gaussian sequence} is a sequence i.i.d. standard Gaussian random variables. Our main goal is to prove the following result.

\begin{theorem}[proven in \secref{mainproof}]\label{thm:rudelson} Given positive integers $d,n\in\N$, let $A_1,\dots,A_n$ be deterministic $d\times d$ Hermitian matrices and $\{\eps_i\}_{i=1}^n$ be either a Rademacher sequence or a standard Gaussian sequence. Define
$Z_n$ as in \eqnref{defZ_n}. Then for all $p\in [1,+\infty)$,
$$\Ex{\|Z_n\|^p}^{1/p}\leq (\sqrt{2\ln(2d)}+C_p)\,\left\|\sum_{i=1}^n A_i^2\right\|^{1/2}$$
where
$$C_p\equiv \left(p\int_{0}^{+\infty}t^{p-1}e^{-\frac{t^2}{2}}\,dt\right)^{1/p}\;(\leq c\sqrt{p}\mbox{ for some universal $c>0$).}$$
\end{theorem}

For $d=1$, this result corresponds to the classical
Khintchine inequalities, which give sub-Guassian bounds for the
moments of $\sum_{i=1}^n\eps_ia_i$ $(a_1,\dots,a_n\in\R)$.
\thmref{rudelson} is implicit in Section $3$ of Rudelson's paper
\cite{Rudelson_RandomIsotropic}, albeit with non-explicit constants. The main Theorem in that paper is the following inequality, which is a simple corollary of \thmref{rudelson}: if $Y_1,\dots,Y_n$ are i.i.d. random (column) vectors in $\C^d$ which are isotropic (i.e $\Ex{Y_1Y_1^*}=I$, the $d\times d$ identity matrix), then:
\begin{equation}\label{eq:rudelsonsbound}\Ex{\left\|\frac{1}{n}\sum_{i=1}^nY_iY_i^* - I\right\|}\leq C\, \Ex{|Y_1|^{\log n}}^{1/\log n}\sqrt{\frac{\log d}{n}}\end{equation} for some universal $C>0$, whenever the RHS of the above inequality is at most $1$. This important result has been applied to several different problems, such as bringing a
convex body to near-isotropic position \cite{Rudelson_RandomIsotropic}; the analysis of for low-rank approximations of matrices \cite{RudelsonVershynin,HalkoEtAl_StochasticLargeMatrices} and graph sparsification \cite{SpielmanSrivastava_SparsificationByEffectiveResistances}; estimating of singular values of matrices with independent rows \cite{MendelsonPajor_SingularValues}; analysing compressive sensing \cite{CandesRomberg_CompressiveSensing}; and related problems in Harmonic
Analysis \cite{Vershynin_RandomFrameErasures,Tropp_ConditioningSubdictionaries}.

The key ingredient of the
original proof of \thmref{rudelson} is a non-commutative Khintchine inequality by Lust-Picard and
Pisier \cite{LustPicardPisier_Khintchine}. This states
that there exists a universal $c>0$ such that for all $Z_n$ as in
the Theorem, all $p\geq 1$ and all $d\times d$ matrices
$\{B_i,D_i\}_{i=1}^n$ with $B_i+D_i=A_i$, $1\leq i\leq n$,
$$\Ex{\left\|Z_n\right\|^p_{S^p}}^{1/p}\leq c\sqrt{p}\,\left(\left\|\sum_{i=1}^n B_iB_i^*\right\|^{1/2}_{S^p}+\left\|\sum_{i=1}^n D_i^*D_i\right\|^{1/2}_{S^p}\right),$$
where $\|\cdot\|_{S^p}$ denotes the {\em $p$-th Schatten norm}:
$\|A\|^p_{S^p}\equiv \tr[(A^*A)^{p/2}].$ \ignore{We note that
this result does not require that the $A_j$ be Hermitian.
It is not hard to deduce Rudelson's lemma from this by taking
suitably large $p$.} Unfortunately, the proof of the
Lust-Picard/Pisier inequality employs language and tools from
non-commutative probability that are rather foreign to most
potential users of \eqnref{rudelsonsbound}.

This note presents an elementary proof of \thmref{rudelson} that
bypasses the above inequality. Our argument is based on an improvement of the
methodology created by Ahlswede and Winter
\cite{AhlswedeWinter_StrongConverse} in order to prove their {\em
operator Chernoff bound}, which also has many applications
e.g. \cite{RussellLandau_AlonRoichman} (the improvement is discussed in \secref{improvement}). This approach only requires elementary
facts from Linear Algebra and Matrix Analysis. The most complicated
result that we use is the Golden-Thompspon inequality
\cite{Golden_GTIneq,Thompson_GTIneq}:
\begin{equation}\label{eq:GTineq}\forall d\in\N,\, \forall \mbox{ $d\times d$ Hermitian matrices }A,B,\,
\tr(e^{A+B})\leq \tr(e^{A}e^{B}).\end{equation} The elementary proof of this classical inequality is sketched in \secref{GTineq}
below.

We have already noted that Rudelson's bound \eqnref{rudelsonsbound} follows simply from \thmref{rudelson}; see \cite[Section 3]{Rudelson_RandomIsotropic} for detais. Here we prove a concentration lemma corresponding to that result under the stronger assumption that $|Y_1|$ is a.s. bounded. While similar results have appeared in other papers
\cite{MendelsonPajor_SingularValues,RudelsonVershynin,Vershynin_RandomFrameErasures}, our proof is simpler and gives explicit (albeit quite large) constants.

\begin{lemma}[Proven in \secref{concentrationHilbert}]\label{lem:concentrationHilbert}Let $Y_1,\dots,Y_n$ be i.i.d. random column vectors in $\C^d$ with $|Y_1|\leq M$ almost surely and $\|\Ex{Y_1Y_1^*}\|\leq 1$. Then:
$$\forall t\geq 0,\Pr{\left\|\frac{1}{n}\sum_{i=1}^nY_iY_i^* - \Ex{Y_1Y_1^*}\right\|\geq t}\leq (2n)^2 e^{-\frac{nt^2}{16M^2+8M^2t}}.$$
\end{lemma}
In particular, a calculation shows that:
$$\left\|\frac{1}{n}\sum_{i=1}^nY_iY_i^* - \Ex{Y_1Y_1^*}\right\|<\eps(n,M)\equiv M\,\sqrt{\frac{72\ln n + 48\ln 2}{n}}\mbox{ with probability }\geq 1-\frac{1}{n}$$
whenever $\eps(n,M)\leq 1$. A key feature both of this Lemma is that the ambient dimension $d$ plays no direct role in the bound. In fact, the same result holds for $Y_i$ taking values in a separable Hilbert space (as in the last section of \cite{MendelsonPajor_SingularValues}).

To conclude the introduction, we present an open problem: is it possible to improve upon Rudelson's bound under further assumptions? There is some evidence that the dependence on $\ln(d)$ in the Theorem, while necessary in general \cite[Remark
3.4]{RudelsonVershynin}, can sometimes be removed. For instance, Adamczak et
al. \cite{AdamczakElAl_ConvexIsotropic} have improved upon
Rudelson's original application of \thmref{rudelson} to convex
bodies, obtaining exactly what one would expect in the
absence of the $\sqrt{\log (2d)}$ term. Another
setting where our bound is a $\theita{\sqrt{\ln d}}$ factor away
from optimality is that of more classical random matrices (cf. the end of
\secref{improvement} below). It would be interesting if one could
sharpen the proof of \thmref{rudelson} in order to reobtain these
results. [Related issues are raised by Vershynin
\cite{Vershynin_NormOfProduct}.]

\section{Preliminaries}

We let $\C^{d\times d}_{\rm Herm}$ denote the set of
$d\times d$ Hermitian matrices, which is a subset of the set $\C^{d\times d}$ of all $d\times d$ matrices with complex entries. The {\em spectral theorem} states
that all $A\in\C^{d\times d}_{\rm Herm}$ have $d$ real eigenvalues
(possibly with repetitions) that correspond to an orthonormal set of
eigenvectors. $\lambda_{\max}(A)$ is the largest eigenvalue of $A$.
The spectrum of $A$, denoted by ${\rm spec}(A)$, is the multiset of
all eigenvalues, where each eigenvalue appears a number of times
equal to its multiplicity. We let
$$\|C\|\equiv \max_{v\in\C^d\,|v|=1}|Cv|$$
denote the operator norm of $C\in\C^{d\times d}$ ($|\cdot|$ is the Euclidean
norm). By the spectral theorem,
$$\forall A\in \C^{d\times d}_{\rm Herm},\,\|A\| = \max\{\lambda_{\max}(A),\lambda_{\max}(-A)\}.$$
Moreover, $\tr(A)$ (the trace of $A$) is the sum of the eigenvalues
of $A$.
\subsection{Spectral mapping} Let $f:\C\to\C$ be an entire analytic
function with a power-series representation $f(z)\equiv \sum_{n\geq 0}c_n\,z^n$ $(z\in\C)$.
If all $c_n$ are real, the expression:
$$f(A)\equiv \sum_{n\geq 0}c_n A^n \;\;(A\in\C^{d\times d}_{\rm
Herm})$$ corresponds to a map from $\C^{d\times d}_{\rm Herm}$ to
itself. We will sometimes use the so-called spectral mapping
property: \begin{equation}\label{eq:spectralmapping}{\rm spec}f(A) =
f({\rm spec}(A)).\end{equation} By this we mean that the eigenvalues
of $f(A)$ are the numbers $f(\lambda)$ with $\lambda\in{\rm
spec}(A)$. Moreover, the multiplicity of $\xi\in{\rm spec}f(A)$ is
the sum of the multiplicities of all preimages of $\xi$ under $f$
that lie in ${\rm spec}(A)$.

\subsection{The positive-semidefinite order}

We will use the notation $A\succeq 0$ to say that $A$ is {\em positive-semidefinite}, i.e. $A\in\C^{d\times d}_{\rm Herm}$ and its eigenvalues are $A$ are non-negative. This is equivalent to saying that $(v,Av)\geq 0$ for all $v\in\C^d$, where $(\cdot,\cdot\cdot)$ is the standard Euclidean inner product.

If $A,B\in\C^{d\times d}_{\rm Herm}$, we write $A\succeq B$ or
$B\preceq A$ to say that $A-B\succeq 0$. Notice that ``$\succeq$" is
a partial order and that:
\begin{equation}\label{eq:isacone}\forall A,B,A',B'\in \C^{d\times d}_{\rm Herm},\, (A\preceq A')\wedge (B\preceq B')\Rightarrow A+A'\preceq B+B'.\end{equation}
Moreover, spectral mapping \eqnref{spectralmapping} implies that:
\begin{equation}\label{eq:square0}\forall A\in \C^{d\times d}_{\rm Herm},\,A^2\succeq 0.\end{equation}
 We will also need the following simple fact.
\begin{proposition}\label{prop:tracemonotone}For all $A,B,C\in\C^{d\times d}_{\rm Herm}:$
\begin{equation}\label{eq:tracemonotone}(C\succeq 0)\wedge (A\preceq
B)\Rightarrow \tr(AC)\leq \tr(BC).\end{equation}\end{proposition}
\begin{proof}To prove this, assume the LHS and observe that the RHS is
equivalent to $\tr(C\Delta)\geq 0$ where $\Delta\equiv B-A$. By
assumption, $\Delta\succeq 0$, hence it has a Hermitian square root
$\Delta^{1/2}$. The cyclic property of the trace implies:
$$\tr(C\Delta) = \tr(\Delta^{1/2}C\Delta^{1/2}).$$
Since the trace is the sum of the eigenvalues, we will be done once
we show that $\Delta^{1/2}C\Delta^{1/2}\succeq 0$. But, since
$\Delta^{1/2}$ is Hermitian and $C\succeq 0$,
$$\forall v\in\C^d,\, (v,\Delta^{1/2}C\Delta^{1/2}v) =
((\Delta^{1/2}v),C(\Delta^{1/2}v))=(w,Cw)\geq 0\mbox{ (with
$w=\Delta^{1/2}v$),}$$ which shows that
$\Delta^{1/2}C\Delta^{1/2}\succeq 0$, as desired.\end{proof}

\subsection{Probability with matrices}
Assume $(\Omega,\sF,\Prwo)$ is
a probability space and $Z:\Omega\to\C^{d\times d}_{\rm Herm}$ is
measurable with respect to $\sF$ and the Borel $\sigma$-field on
$\C^{d\times d}_{\rm Herm}$ (this is equivalent to requiring that
all entries of $Z$ be complex-valued random variables). $\C^{d\times d}_{\rm Herm}$ is a metrically complete vector space and one can naturally define an expected value $\Ex{Z}\in\C^{d\times
d}_{\rm Herm}$. This turns out to be the matrix $\Ex{Z}\in\C^{d\times d}_{\rm Herm}$ whose $(i,j)$-entry is the
expected value of the $(i,j)$-th entry of $Z$. [Of course, $\Ex{Z}$ is only defined if all entries of $Z$ are integrable, but this will always be the case in this paper.]

The definition of expectations implies that traces and expectations
commute:
\begin{equation}\label{eq:traceex}\tr(\Ex{Z})=\Ex{\tr(Z)}.\end{equation}
Moreover, one can check that the usual product rule is satisfied:
\begin{equation}\label{eq:productrule}\mbox{If $Z,W:\Omega\to\C^{d\times d}_{\rm
Herm}$ are measurable and independent, }
\Ex{ZW}=\Ex{Z}\Ex{W}.\end{equation}

\section{Proof of \thmref{rudelson}}\label{sec:mainproof}

\begin{proof}[of \thmref{rudelson}] We wish to control the tail behavior of:
$$\|Z_n\|=\max\{\lambda_{\max}(Z_n),\lambda_{\max}(-Z_n)\}.$$
However, $Z_n$ and $-Z_n$ have the same distribution. It follows that:
\begin{equation*}\forall t\geq 0,\, \Pr{\|Z_n\|\geq t}\leq 2\,\Pr{\lambda_{\max}(Z_n)\geq t}.\end{equation*}
The usual Bernstein trick implies that for all $t\geq 0$,
\begin{equation*}\forall t\geq 0,\,\Pr{\lambda_{\max}(Z_n)\geq
t}\leq\inf_{s>0}e^{-st}\Ex{e^{s\lambda_{\max}(Z_n)}}.\end{equation*}
The function ``$x\mapsto e^{sx}$" is monotone non-decreasing and positive for all $s\geq 0$. It follows from the spectral mapping property \eqnref{spectralmapping} that for all $s\geq 0$, the largest eigenvalue of $e^{sZ_n}$ is $e^{s\lambda_{\max}(Z_n)}$ and all eigenvalues of $e^{sZ_n}$ are non-negative. Using the equality ``trace $=$ sum of eigenvalues" implies that for all $s\geq 0$,
$$\Ex{e^{s\lambda_{\max}(Z_n)}} = \Ex{\lambda_{\max}\left(e^{sZ_n}\right)}\leq \Ex{\tr\left(e^{sZ_n}\right)}.$$
As a result, we have the inequality:
\begin{equation}\label{eq:justlikeAW}\forall t\geq 0,\, \Pr{\|Z_n\|\geq t}\leq 2\inf_{s\geq 0}e^{-st}\Ex{\tr\left(e^{sZ_n}\right)}.\end{equation}
Up to now, our proof has followed Ahlswede and Winter's argument. The next lemma, however, will require new ideas.
\begin{lemma}\label{lem:main}For all $s\in\R$,
$$\Ex{\tr(e^{sZ_n})}\leq \tr\left(e^{\frac{s^2\sum_{i=1}^n A_i^2}{2}}\right).$$\end{lemma}
This lemma is proven below. We will now show how it implies Rudelson's bound. Let
$$\sigma^2\equiv \left\|\sum_{i=1}^n A_i^2\right\| = \lambda_{\max}\left(\sum_{i=1}^n A_i^2\right).$$
[The second inequality follows from $\sum_{i=1}^n A_i^2\succeq 0$, which holds because of \eqnref{isacone} and \eqnref{square0}.] We note that:
$$\tr\left(e^{\frac{s^2\sum_{i=1}^n A_i^2}{2}}\right)\leq d\, \lambda_{\max}\left(e^{\frac{s^2\sum_{i=1}^n A_i^2}{2}}\right) = d\, e^{\frac{s^2\sigma^2}{2}}$$
where the equality is yet another application of spectral mapping \eqnref{spectralmapping} and the fact that ``$x\mapsto e^{s^2x/2}$" is monotone increasing. We deduce from the Lemma and \eqnref{justlikeAW} that:
\begin{equation}\label{eq:notjustAW}\forall t\geq 0,\, \Pr{\|Z_n\|\geq t}\leq 2d\,\inf_{s\geq 0}e^{-st+\frac{s^2 t^2}{2}} = 2d\,e^{-\frac{t^2}{2\sigma^2}}.\end{equation}
This implies that for any $p\geq 1$,
\begin{eqnarray*}\frac{1}{\sigma^p}\Ex{(\|Z_n\|-\sqrt{2\ln(2d)}\sigma)_+^p} &=& \ignore{ p\int_{0}^{+\infty}t^{p-1}\Pr{\frac{1}{\sigma}(\|Z_n\|-\sqrt{2\ln(2d)}\sigma)_+\geq t}\,dt\\ &=&} p\int_{0}^{+\infty}t^{p-1}\Pr{\|Z_n\|\geq (\sqrt{2\ln(2d)}+t)\sigma}\,dt\\ \mbox{(use\eqnref{notjustAW})}&\leq& 2pd\int_{0}^{+\infty} t^{p-1}e^{-\frac{(t+\sqrt{2\ln(2d)})^2}{2}}\,dt\\ &\leq & 2pd\int_{0}^{+\infty} t^{p-1}e^{-\frac{t^2+2\ln(2d)}{2}}\,dt\;=C^p_p\end{eqnarray*}
Since $0\leq \|Z_n\|\leq \sqrt{2\ln(2d)}\sigma +
(\|Z_n\|-\sqrt{2\ln(2d)}\sigma)_+,$ this implies the $L^p$ estimate
in the Theorem. The bound ``$C_p\leq c\sqrt{p}$" is standard and we
omit its proof.\end{proof}

To finish, we now prove \lemref{main}.

\begin{proof}[of \lemref{main}] Define $D_{0}\equiv\sum_{i=1}^n s^2A_i^2/2$ and
$$D_{j} \equiv D_0 +  \sum_{i=1}^j \left(s\eps_iA_i - \frac{s^2A_i^2}{2}\right)\;\; (1\leq j\leq n).$$We will
prove that for all $1\leq j \leq n$:
\begin{equation}\label{eq:decreaseswithj}\Ex{\tr\left(\exp\left(D_j\right)\right)}\leq
\Ex{\tr\left(\exp\left(D_{j-1}\right)\right)}.\end{equation} Notice
that this implies $\Ex{\tr(e^{D_n})}\leq \Ex{\tr(e^{D_0})}$, which is the precisely the Lemma. To prove \eqnref{decreaseswithj}, fix $1\leq j\leq n$. Notice that
$D_{j-1}$ is independent from $s\eps_jA_j-s^2A_j^2/2$ since the
$\{\eps_i\}_{i=1}^n$ are independent. This implies that:
\begin{eqnarray*}\Ex{\tr\left(\exp\left(D_j\right)\right)} &=&
\Ex{\tr\left(\exp\left(D_{j-1} + s\eps_jA_j - \frac{s^2A_j^2}{2}\right)\right)}\\
\mbox{(use Golden-Thompson \eqnref{GTineq})}&\leq &
\Ex{\tr\left(\exp\left(D_{j-1}\right)
\exp\left(s\eps_jA_j-\frac{s^2A_j^2}{2}\right)\right)}\\
\mbox{($\tr(\cdot)$ and $\Ex{\cdot}$ commute, \eqnref{traceex})}
 &=&
\tr\left(\Ex{\exp\left(D_{j-1}\right)\exp\left(s\eps_jA_j-\frac{s^2A_j^2}{2}\right)}\right).\\
\mbox{(use product rule, \eqnref{productrule})}
 &=& \tr\left(\Ex{\exp\left(D_{j-1}\right)}\Ex{\exp\left(s\eps_jA_j-\frac{s^2A_j^2}{2}\right)}\right).\end{eqnarray*}

By the monotonicity of the trace \eqnref{tracemonotone} and the fact
that $\exp\left(D_{j-1}\right)\succeq 0$ (which follows from
\eqnref{spectralmapping}), we will be done once we show that:
\begin{equation}\label{eq:finall}\Ex{\exp\left(s\eps_jA_j-\frac{s^2A_j^2}{2}\right)}\preceq
I.\end{equation} The key fact is that $s\eps_j A_j$ and $-s^2A_j^2/2$
{\em always} commute, hence the exponential of the sum is the
product of the exponentials. Applying \eqnref{productrule} and
noting that $e^{-s^2A_j^2/2}$ is constant, we see that:
$$\Ex{\exp\left(s\eps_jA_j-\frac{s^2A_j^2}{2}\right)}=\Ex{\exp\left(s\eps_jA_j\right)}e^{-\frac{s^2A_j^2}{2}}.$$
 In the Gaussian case, an explicit calculation shows that
$\Ex{\exp\left(s\eps_jA_j\right)}=e^{s^2A^2_j/2}$, hence \eqnref{finall}
holds. In the Rademacher case, we have:
$$\Ex{\exp\left(s\eps_jA_j\right)}e^{-\frac{s^2A_j^2}{2}}=f(A_j)$$
where $f(z) = \cosh(sz)e^{-s^2z^2/2}$. It is a classical fact that
$0\leq \cosh(x)\leq e^{x^2/2}$ for all $x\in\R$ (just compare the
Taylor expansions); this implies that $0\leq f(\lambda)\leq 1$ for
all eigenvalues of $A_j$. Using spectral mapping
\eqnref{spectralmapping}, we see that:
$${\rm spec}f(A_j)=f({\rm spec}(A_j))\subset [0,1],$$
which implies that $f(A_j)\preceq I$. This proves \eqnref{finall} in
this case and finishes the proof of \eqnref{decreaseswithj} and of the Lemma.\end{proof}

\subsection{Remarks on the original AW approach}\label{sec:improvement}

A direct adaptation of the original argument of Ahlswede and Winter \cite{AhlswedeWinter_StrongConverse} would lead to an inequality of the form:
$$\Ex{\tr(e^{sZ_n})}\leq \tr\left(\Ex{e^{s\eps_n A_n}}\Ex{e^{sZ_{n-1}}}\right).$$
One sees that:
$$\Ex{e^{s\eps_n A_n}}\preceq e^{\frac{s^2A_n^2}{2}}\preceq e^{\frac{s^2\|A_n^2\|}{2}}\,I.$$
However, only the second inequality seems to be useful, as there is no obvious relationship between
$$\tr\left(e^{\frac{s^2A_n^2}{2}}\Ex{e^{sZ_{n-1}}}\right)$$
and
$$\tr\left(\Ex{e^{s\eps_{n-1}A_{n-1}}}\Ex{e^{sZ_{n-2}+\frac{s^2A_n^2}{2}}}\right),$$
which is what we would need to proceed with induction. [Note that Golden-Thompson \eqnref{GTineq} cannot be undone and fails for three summands, \cite{Thompson_GTIneq}.] The best one can do with the second inequality is:
$$\Ex{\tr(e^{sZ_n})}\leq d\,e^{\frac{s^2\sum_{i=1}^n\|A_i\|^2}{2}}.$$
This would give a version of \thmref{rudelson} with
$\sum_{i=1}^n\|A_i\|^2$ replacing $\|\sum_{i=1}^nA_i^2\|$. This modified result is always worse than the actual Theorem, and can be dramatically so. For instance, consider
the case of a {\em Wigner matrix} where:
$$Z_{n}\equiv \sum_{1\leq i\leq j\leq m}\eps_{ij}A_{ij}$$
with the $\eps_{ij}$ i.i.d. standard Gaussian and each $A_{ij}$ has
ones at positions $(i,j)$ and $(j,i)$ and zeros elsewhere (we take
$d=m$ and $n=\binom{m}{2}$ in this case). Direct calculation
reveals:
$$\left\|\sum_{ij} A_{ij}^2\right\| = \left\|(m-1)I\right\| = m-1\ll
\binom{m}{2} = \sum_{ij} \|A_{ij}\|^2.$$

We note in passing that neither approach is sharp in this case, as $\|\sum_{ij}\eps_{ij}A_{ij}\|$ concentrates around $2\sqrt{m}$ \cite{FurediKomlos_RandomMatrices}.

\section{Concentration for rank-one operators}\label{sec:concentrationHilbert}
In this section we prove \lemref{concentrationHilbert}.

\begin{proof}[of \lemref{concentrationHilbert}] Let
$$\phi(s) \equiv \Ex{\exp\left(s\left\|\frac{1}{n}\sum_{i=1}^nY_iY_i^* -
\Ex{Y_1Y_1^*}\right\|\right)}.$$ We will show below that:
\begin{equation}\label{eq:goalconcentration}\forall
s\geq 0,\, \phi(s) \leq 2n\,e^{2M^2s^2/n} \phi(2M^2s^2/n).\end{equation}
By Jensen's inequality, $\phi(2Ms^2/n)\leq \phi(s)^{2M^2s/n}$ whenever
$2M^2s/n\leq 1$, hence \eqnref{goalconcentration} implies:
$$\forall 0\leq s\leq n/2M^2,\, \phi(s) \leq (2n)^{\frac{1}{1-2M^2s/n}}
e^{\frac{2M^2s^2}{n-2M^2s}}.$$ Since
$$\forall s\geq 0,\, \Pr{\left\|\frac{1}{n}\sum_{i=1}^nY_iY_i^* -
\Ex{Y_1Y_1^*}\right\|\geq t}\leq e^{-st}\phi(s),$$ the Lemma then
follows from the choice
$$s\equiv \frac{tn}{8M^2 + 4M^2t}$$
and a few simple calculations. [Notice that $2M^2s/n\leq 1/2$ with this choice, hence $1/(1-2M^2s/n)\leq 2$.]

To prove \eqnref{goalconcentration}, we begin with symmetrization
(see e.g. \cite{LedouxTalagrand_BanachSpaces}):
$$\phi(s)\leq
\Ex{\exp\left(2s\left\|\frac{1}{n}\sum_{i=1}^n\eps_iY_iY_i^*\right\|\right)},$$
where $\{\eps_i\}_{i=1}^n$ is a Rademacher sequence independent of
$Y_1,\dots,Y_n$. Let $\sS$ be the (random) span of $Y_1,\dots,Y_n$ and
$\tr_{\sS}$ denote the trace operation on linear operators mapping
$\sS$ to itself. Following the argument in \thmref{rudelson}, we
notice that:
$$\Ex{\exp\left(2s\left\|\frac{1}{n}\sum_{i=1}^n\eps_iY_iY_i^*\right\|\right)\mid
Y_1,\dots,Y_n}\leq
2\Ex{\tr_{\sS}\left\{\exp\left(\frac{2s}{n}\sum_{i=1}^n\eps_iY_iY_i^*\right)\right\}\mid
Y_1,\dots,Y_n}.$$ \lemref{main} implies:
\begin{eqnarray*}\Ex{\exp\left(2s\left\|\frac{1}{n}\sum_{i=1}^n\eps_iY_iY_i^*\right\|\right)\mid
Y_1,\dots,Y_n} &\leq& 2\tr_{\sS}\left\{\exp\left(\frac{2s^2}{n^2}\sum_{i=1}^n(Y_iY_i^*)^2\right)\right\}\\
&\leq &
2n\exp\left(\left\|\frac{2s^2}{n^2}\sum_{i=1}^n(Y_iY_i^*)^2\right\|\right),\end{eqnarray*}
using spectral mapping \eqnref{spectralmapping}, the equality ``trace $=$ sum of eigenvalues" and the fact that
$\sS$ has dimension $\leq n$. A quick calculation shows that $0\preceq (Y_iY_i^*)^2 =
|Y_i|^2\,Y_iY_i^* \preceq M^2Y_iY_i^*$, hence \eqnref{isacone} implies:
$$0\preceq \frac{2s^2}{n^2}\sum_{i=1}^n(Y_iY_i^*)^2\preceq
\frac{2M^2s^2}{n}\, \left(\frac{1}{n}\sum_{i=1}^n Y_iY_i^*\right).$$
Therefore:
$$\left\|\frac{2s^2}{n^2}\sum_{i=1}^n(Y_iY_i^*)^2\right\| \leq
\frac{2M^2s^2}{n}\left\|\frac{1}{n}\sum_{i=1}^n
Y_iY_i^*\right\|\leq
\frac{2M^2s^2}{n}\left\|\frac{1}{n}\sum_{i=1}^n Y_iY_i^* -
\Ex{Y_1Y_1^*}\right\| + \frac{2M^2s^2}{n}.$$ [We used
$\|\Ex{Y_1Y_1^*}\|\leq 1$ in the last inequality.] Plugging this into the conditional
expectation above and integrating, we obtain \eqnref{goalconcentration}:
$$\phi(s)\leq 2n\, \Ex{\exp\left(\frac{2M^2s^2}{n}\left\|\frac{1}{n}\sum_{i=1}^n Y_iY_i^* -
\Ex{Y_1Y_1^*}\right\| + \frac{2M^2s^2}{n}\right)} =
2ne^{2M^2s^2/n}\,\phi(2M^2s^2/n).$$\end{proof}

\section{Proof sketch for Golden-Thompson
inequality}\label{sec:GTineq}

As promised in the Introduction, we sketch an elementary proof of inequality \eqnref{GTineq}. We will need the {\em Trotter-Lie formula}, a simple consequence of the Taylor formula for $e^X$:
\begin{equation}\label{eq:trotter}\forall A,B\in\C^{d\times d}_{\rm Herm},\, \lim_{n\to +\infty}(e^{A/n}e^{B/n})^n =
e^{A+B}.\end{equation}The second ingredient is the inequality:
\begin{equation}\label{eq:decreasing}\forall k\in\N, \forall X,Y\in \C^{d\times d}_{\rm Herm}\,:\, X,Y\succeq 0 \Rightarrow \tr((XY)^{2^{k+1}})\leq \tr((X^2Y^2)^{2^{k}}).\end{equation}
This is proven in of \cite{Golden_GTIneq} via an argument using the existence of positive-semidefinite square-roots for
positive-semidefinite matrices, and the Cauchy-Schwartz inequality
for the standard inner product over $\C^{d\times d}$. Iterating \eqnref{decreasing}
implies:
\begin{equation*}\forall X,Y\in \C^{d\times d}_{\rm Herm}\,:\, X,Y\succeq 0 \Rightarrow \tr((XY)^{2^{k}})\leq \tr(X^{2^k}Y^{2^k}).\end{equation*}
Apply this to $X=e^{A/2^k}$ and $Y=e^{B/2^k}$ with
$A,B\in\C^{d\times d}_{\rm Herm}$. Spectral mapping \eqnref{spectralmapping} implies $X,Y\succeq 0$ and we deduce:
\begin{equation*}\tr((e^{A/2^k}e^{B/2^k})^{2^{k}})\leq \tr(e^Ae^B).\end{equation*}
Inequality \eqnref{GTineq} follows from letting $k\to +\infty$,
using \eqnref{trotter} and noticing that $\tr(\cdot)$ is continuous.

\end{document}